\documentclass{colt2012} 

\usepackage{dsfont}

\def\C{\mathbb C}
\def\E{\mathbb E}
\def\R{\mathbb R}

\def\P{\mathbb P}

\def\Z{\mathbb Z}

\newtheorem{dfn}{Definition}

\newtheorem{lem}{Lemma}

\newtheorem{thm}{Theorem}

\newtheorem{rmk}{Remark}
\newtheorem{exm}{Example}

\title[Fast rates in learning with dependent observations]
      {Fast rates in learning with dependent observations}

 \coltauthor{\Name{Pierre Alquier} \Email{alquier@math.univ-paris-diderot.fr}\\
 \addr CREST-LS \& LPMA, Universit\'e Paris 7, 175, rue du Chevaleret, 75013 Paris,
  FRANCE
 \AND
 \Name{Olivier Wintenberger} \Email{wintenberger@ceremade.dauphine.fr}\\
 \addr CREST-LFA \& CEREMADE, Universit\'e Paris Dauphine, Place du Maréchal De Lattre
De Tassigny, 75775 PARIS CEDEX 16, FRANCE
 }

\begin{document}

\maketitle

\begin{abstract}
In this paper we tackle the problem of fast rates in time series forecasting
from a statistical learning perspective.
In a serie of papers (e.g. \cite{meir,modha,alqwin}) it is shown that the main
tools used in learning theory with iid observations can be extended to the
prediction of time series.
The main message of these papers is that, given a family of predictors, 
we are able to build a new predictor that predicts the series as well as the
best predictor in the family, up to a remainder of order $1/\sqrt{n}$.
It is known that this rate cannot be improved in general. In this paper, we
show that in the particular case of the least square loss, and under a strong assumption
on the time series ($\phi$-mixing) the remainder is actually of order $1/n$.
Thus, the optimal rate for iid variables, see e.g. \cite{TsybakovAgg}, and 
individual sequences, see \cite{lugosi} is, for the first time, achieved for
uniformly mixing processes.
We also show that our method is optimal for aggregating sparse linear combinations
of predictors.
\end{abstract}


\section{Intro}

The problem of time series forecasting is a standard problem in statistics.
The parametric approach contains a wide range of models associated with efficient
estimation
and prediction methods, see e.g. \cite{hamilton,davis}.

In the last few years, several universal approaches emerged from various fields
such that non-parametric statistics, machine learning, computer science and game
theory. These approaches share some common features: the aim is to
to build a prediction procedure that is able to predict the
series as well
as the best predictor in a given set of initial predictors, say $\Theta$.
The set of predictors are usually inspired by different parametric or non-parametric
statistical models. The true
distribution of the data is
not assumed to belong to one of these models. However, we can distinguish two
classes in these approaches, with different quantification of the objective,
and different terminologies:
\begin{itemize}
\item in the ``prediction of individual sequences'' approach, predictors are
usually
called ``experts''. The objective is online prediction: at each date $t$, a
prediction of the future realization $x_{t+1}$ is based on the previous observations
$x_1$, ..., $x_t$,
the objective being to minimize the cumulative prevision loss. See for example
\cite{lugosi,stoltz} for an introduction.
\item in the statistical learning approach, the given predictors are sometimes
referred as ``models'' or ``concepts''. The batch setting is more classical in statistics.
A prediction procedure is build on a complete sample $X_1$, ..., $X_n$. The performance
of the procedure is compared on average with the best predictor, called the ``oracle''.
The environment is not deterministic and some hypotheses like mixing or weak dependence
is required: see \cite{meir,modha,alqwin}.
\end{itemize}

In both settings, we are able to predict a bounded time series as well as
the best expert, up to a small remainder. This type of results is referred in
statistical
 theory as an oracle inequality.
In general, neglecting the size of the set of predictors $\Theta$, the remainder
is of the
order $1/\sqrt{n}$ in both approaches: see, e.g., \cite{lugosi} for the
``individual sequences''
approach; for the ``statistical learning approach''
the rate $1/\sqrt{n}$ is reached in \cite{alqwin}. This paper is based on
the following remark:
in the case of prediction of
individual sequences, under stronger assumption on the loss function (satisfied
e.g. by the
quadratic loss), a fast rate $1/n$ can be reached. Note that \cite{meir,modha}
deal with the quadratic loss, their rate can be better than $1/\sqrt{n}$ but
cannot reach $1/n$. Here, we prove that the same result is true in the statistical
learning setting. Namely, under a $\phi$-mixing assumption
introduced in
\cite{ib62}, we are able to reach the fast rate in the batch setting for the
quadratic loss.

Following \cite{alqwin},  we will use tools from the
PAC-Bayesian theory
to build our prediction procedure. Historically, the PAC-Bayesian point of view
emerged in
statistical learning to deal with supervised classification (using the
$0/1$-loss), see the
seminal papers \cite{STW97,McA2}. These results were extended to general loss
functions
and more accurate bounds were then given, see for example
\cite{Catoni2004,Catoni2007,
AlquierPAC,DT1,AudibertHDR,AL,Seldin,gerchi}. Interestingly enough, PAC-Bayesian
methods often lead to a prediction
procedure that is an aggregation of the various predictors in $\Theta$ with
exponential
weights, a standard procedure in individual sequences prediction (introduced by
\cite{VOVK,LiWa}).  It is striking to note that this procedures receives
theoretical justification from approaches that have so different philosophies
and objectives.
This procedures received various names: EWA, for Exponentially Weighted
Aggregate, in \cite{DT1,gerchi},
Gibbs estimator in \cite{Catoni2004,Catoni2007,AlquierPAC,AudibertHDR}, weighted
majority
algorithm in \cite{LiWa}... In \cite{AudibertPhD}, it is also proved that this
estimator is simply the Bayesian estimator under
suitable model and prior.

In Section~\ref{section_context} we introduce the notations used in the
whole paper, in particular the time series $(X_t)_{t\in\mathbb{Z}}$ and the set
of predictors $\Theta$.
Section~\ref{section_description} is devoted to the description of the Gibbs
estimator.
Our main result is Theorem \ref{mainthm},
it is stated in
Section~\ref{section_theorem}. In Section~\ref{sectionexamples} we provide examples
of time series
satisfying the main assumption of Theorem \ref{mainthm}
($\phi$-mixing). In Section~\ref{section_algo} we
discuss the implementation
of our procedure using MCMC methods and show the results of some simulations.
Finally,
proofs are given Section~\ref{sectionproofs}, with some technical results postponed
to the appendix.
As we will see, the main tool needed to apply PAC-Bayesian techniques is a
control of the Laplace transform
of the prevision risk. In the iid setting, this might be done using classical
Hoeffding's or
Bernstein's Inequalities. In the context of $\phi$-mixing, such a result is
provided
by a powerful result in \cite{samson}.

Note that in this paper, we focus on the case where the set of predictors is
the linear span of a finite family of basic predictors. Theorem \ref{mainthm}
will be of particular interest in the case where a sparse combination of those
basic predictors provide a good prediction. But the results in these paper can
be extended in other contexts (e.g. if we only want to predict as well as the
best basic predictor). The proof of Theorem \ref{mainthm} involves a general
result, Lemma \ref{PACBAYES}, that can be adapted to these various context.

\section{The context}
\label{section_context}

\subsection{The observation}

We assume that we observe $(X_1,\dots,X_n)$ where $(X_{t})_{t\in\mathbb{Z}}$
is a real, stationary process, bounded by a constant $B$. We remind the
$\phi$-mixing
coefficients of the process $(X_t)$ as introduced by
\cite{ib62}:
\begin{dfn}[$\phi$-mixing coefficients] We define the $\phi$-mixing coefficients
of the
process $(X_{t})_{t\in\mathbb{Z}}$ by
$$
\phi_r=\sup_{(A,B)\in\,\mathfrak{S}_0\times\mathfrak{F}_r}|\pi(B/A)-\pi(B)|
$$
where $\mathfrak{S}_0=\sigma(X_t, t\le 0)$ and $\mathfrak{F}_r=\sigma(X_t,t\ge
r)$.
We also define:
$$ K_{\phi}^{(n)} (q):=1+\sum_{r=1}^{n-q}\sqrt{\phi_{\lfloor r/q \rfloor}}. $$
\end{dfn}

\subsection{Set of predictors}

We set a value $q$ and a family of functions: $g_{1}$, ..., $g_{p}:[-B,B]^{q}
\rightarrow [-B,B] $. The set of predictors, for a given $b>0$, is defined by:
$$ \left\{f_{\theta},\theta\in\Theta(b) \right\} $$
where $\Theta(b)=\{\theta\in\R^{p}:\|\theta\|_{1}< b\}$, and
$$ f_{\theta} = \sum_{j=1}^{p}\theta_{j} g_{j} .$$
We also put $\Theta=\R^{p}$ and our objective is to find a $\theta$ such
that $X_{q+1}$ is well predicted by $f_{\theta}(X_{q},...,X_{1})$ on average under the stationary distribution.

Note that we will allow very large set of predictors (experts, ...). Actually,
we will allow $n \ll p$. In this case, a sparsity assumption will be necessary:
namely, it is possible to build a good predictor $\theta$ such that most of
its coordinates are close to $0$. This is now a classical assumption in statistical
learning theory, see e.g. \cite{Tibshirani-LASSO,vdgb}.

\begin{exm}[Auto-regressive predictors]
\label{examar}
A very classical example is to design predictors based on auto-regressive models
(AR).
We put $p=q$ and $g_{1}(x_{q},...,x_{1})=x_{q}$, ...,
$g_{q}(x_{q},...,x_{1})=x_{1}$
so we obtain AR predictors
$$ f_{\theta}(X_{q},...,X_{1}) = \sum_{j=1}^{q}\theta_{j} X_{p-j} .$$
Note that in this case, $p<n$.
\end{exm}

\begin{exm}
We can extend the previous setting to non-linear AR predictors. For example,
We take $p=2^q$ and $ g_{1}(x_{q},...,x_{1}) = 1(x_{q}>0,...,x_{1}>0) $, then
$g_{2}(x_{q},...,x_{1}) = 1(x_{q}>0,...,x_{2}>0,x_{1}\leq 0) $, ..., up to
$ g_{2^{q}}(x_{q},...,x_{1}) = 1(x_{q}\leq 0,...,x_{1}\leq 0) $.
\end{exm}

\begin{dfn}[Prevision and empirical risks]
We define the prevision risk
$$ R(\theta) = \E_{\P}
\left\{\left[X_{q+1}-f_{\theta}(X_{q},...,X_{1})\right]^{2}\right\} $$
and the empirical risk
$$ r(\theta) = \frac{1}{n-q} \sum_{i=q+1}^{n}
\left[X_{i}-f_{\theta}(X_{i-1},...,X_{i-q})\right]^{2} $$
and
$$ \overline{\theta}\in\arg\min_{\Theta}R .$$
\end{dfn}

The objective is to build an estimator $\hat{\theta}$ based on the observations
$(X_1,\ldots,X_n)$ such that $R(\hat{\theta})$
is as small as possible. We see in the next sections that the Gibbs estimator
reaches this objective.

\section{Description of the method}
\label{section_description}

Ths Gibbs estimator as defined in \cite{Catoni2007} requires a prior
distribution
on the parameter space.

\begin{dfn}[The prior]
For $I\subset\{1,...,p\}$, $b>0$,
$$ \Theta_{I}(b) = \biggl\{\theta\in\Theta(b):
         \quad \forall i\notin I, \theta_{i} = 0\biggr\} $$
and
$$ \Theta_{I} = \biggl\{\theta\in\Theta:
         \quad \forall i\notin I, \theta_{i} = 0\biggr\} .$$
Finally, let us put $\pi_{b}^{I}$ the uniform probability measure on
$\Theta_{I}(b+1) $.
We put, for some $b>0$,
$$ \pi_{b} \propto \sum_{k=0}^{n} 2^{-k-1}
\sum_{
\tiny{
\begin{array}{c}
I\subset\{1,...,p\}
\\
|I|=k
\end{array}
}}
{p\choose k}^{-1} \pi_{b}^{I}.
$$
\end{dfn}

Remark that in order to predict as well as the best predictor in $\Theta(b)$,
the prior
distribution has to be defined on $\Theta(b+1)$, for technical reasons that will
become
clear in the proofs (see the appendix). We are now ready to give the definition
of the Gibbs
estimator.

\begin{dfn}[Gibbs estimator]
We define, for any $b>0$ and $\lambda>0$, $ \hat{\rho}_{\lambda,b} $ such that
$$ \frac{d\hat{\rho}_{\lambda,b}}{d\pi_{b}}(\theta)
     = \frac{\exp\left[-\lambda
r(\theta)\right]}{\int_{\Theta(b)}\exp[-\lambda r] d\pi_{b}},$$
and we put
\begin{equation}
\label{gibbsdfn}
\hat{\theta}_{\lambda,b} = \int_{\Theta(b)}\theta
\hat{\rho}_{\lambda}(d\theta) .
\end{equation}
\end{dfn}

The parameter $\lambda$ is called the inverse temperature parameter. Its choice
is a problem in practice, see the discussions in
\cite{Catoni2003,Catoni2004,Catoni2007,AlquierPAC}.
In theory, we will see that $\lambda$ of the order $n$ will lead to fast rates
for prediction. In practice, $\lambda = n/\hat{{\rm var}}(X)$ leads to
satisfying results in
our simulations, where $\hat{{\rm var}}(X)$ is the empirical variance of the
observed time series.
The practical computation of $\hat{\theta}_{\lambda,b}$ can also be a problem.
In \cite{DT1} a Langevin Monte-Carlo algorithm is used. Here, as in \cite{AL},
the Reversible Jump
MCMC of \cite{RJMCMC} is used, see Section~\ref{section_algo}.

\section{Theoretical results}
\label{section_theorem}

\begin{thm}[Oracle inequality for the Gibbs estimator]
\label{mainthm}
Assume that $\|\overline{\theta}\|_{1}<b$ and that there exists a
constant $\Phi(q)$ such that for any $n\in\mathbb{N}$, $ \Phi(q) \geq
K^{(n)}_{\phi}(q)$.
Choose
$$\eta\in\left(0, \frac{16}{\Phi(q)} \right]\quad\mbox{and}\quad\lambda =
\frac{\eta(n-q)}{64 \Phi(q) (2+b)^2 B^2 } .$$
We have, with probability at least $1-\varepsilon$ on the drawing of the sample
$(X_1,\cdots,X_n)$,
\begin{multline*}
R(\hat{\theta}_{\lambda,b}) - R(\overline{\theta})
\leq
\inf_{
\tiny{
\begin{array}{c}
I\subset\{1,...,p\}
\\
|I|< \frac{\eta(n-q)}{32\Phi(q) (2+b)^2}
\\
\theta \in \Theta_{I}(b)
\end{array}
}
}
\Biggl\{
\left(\frac{2+\eta}{2-\eta}\right) \Bigl(R(\theta) - R(\overline{\theta}) \Bigr)
\\
+ \frac{64\Phi(q)(2+b)^2B^2}{(n-q)\eta}\left[|I| \left( B +
2 \log\left( \frac{Bbp{\rm e}}{|I|}\sqrt{\frac{2\eta
(n-q)}{|I|}}\right)\right)+2\log\left(\frac{2}{\varepsilon}\right)\right].
\Biggr\}
\end{multline*}
\end{thm}

The full proof is given in the appendix. In order to understand
this result, it is particularly useful to think of a particular case where
there is a sparse optimal predictor: we assume that there is a
$\overline{\theta}\in
\arg\min_{\Theta(b)} R$ that has only a few number $p_{0}$ of non-zero
coordinates.
This is the classical ``sparsity'' assumption. Then in this case, taking
$\theta=\overline{\theta}$
in the previous result leads to
\begin{multline}
\label{corollary}
R(\hat{\theta}_{\lambda,b}) - R(\overline{\theta})
\\
\leq \frac{64\Phi(q)(2+b)^2B^2}{(n-q)\eta}\left\{ p_{0} \left[ B +
2 \log\left( \frac{Bbp{\rm e}}{p_{0}}\sqrt{\frac{2\eta
(n-q)}{p_{0}}}\right)\right]
+2\log\left(\frac{2}{\varepsilon}\right)\right\}
\end{multline}
for $n$ large enough - actually, $n> q+ p_{0}[32 \Phi(q)
(2+b)^2]/\eta$.
We obtain that this is not the true dimension $p$ of $\Theta$
that determines a rate $p/n$, but the intrinsic dimension $p_{0}$ of
$\overline{\theta}$
as the rate is $p_{0}\log(pn)/n$. With iid observations, \cite{DT1,AL} obtained
the same
result, with rate $p_{0}\log(p)/n$. In \cite{gerchi}, the same rate is reached
in the context of prediction of individual sequences.

Note that of course the strength of Theorem \ref{mainthm} when compared to
Inequality~\ref{corollary}
is that it ensures that $\hat{\theta}_{\lambda,b}$ will give good prediction not
only when
$\overline{\theta}$ is sparse, but also when it can only be approximated by a
sparse parameter $\theta$.

\begin{rmk}
The value of $\lambda$ proposed in the Theorem depends on the $\phi$-mixing
coefficients
of the time series. Of course, these coefficients are unknown. One can check in
the proof of
Theorem \ref{mainthm} that any $\lambda$ of the order of $n$ would lead to the
same rate of
convergence, but with less precise constants. However, in practice, this does
not tell us
how to calibrate $\lambda$. It is of course possible to use a procedure such as
cross-validation.
However, in \cite{DT1} or \cite{AL}, it is observed that the value $\lambda =
n/(4\sigma^2)$
or $\lambda = n/(2\sigma^2)$, where $\sigma^2$ is the variance of then noise,
performs well in practice, 
and receives a theoretical justification in the iid setting. So we propose here
the heuristic
value $\lambda = n/\hat{{\rm var}}(X)$ leads to satisfying results in
our simulations, where $\hat{{\rm var}}(X)$ is the empirical variance of the
observed time series.
We will see in Section~\ref{section_algo} that it performs well on a
set of simulations.
\end{rmk}

\section{Some examples of $\phi$-mixing processes}
\label{sectionexamples}

In this section we study the behavior of the prediction procedure on some
classical $\phi$-mixing processes. In all the section $(\epsilon_t)$ denotes an
iid sequence called the innovations.
\subsection{The AR($p$) model}
We consider the case where the observations $(X_t)$ satisfy an AR($p$) model:
\begin{equation}
X_t=\sum_{j=1}^pa_{j}X_{t-j}+\epsilon_t, \qquad\forall t\in\Z.
\end{equation}
Here both $p\in\{1,2, \ldots\} $ and $(a_j)$ are unknown, $(\epsilon_t)$ is
bounded with a distribution possessing an absolutely continuous component.
Assume that $\mathcal A(z)=\sum_{j=1}^pa_jz^j$ has no root inside the unit disk
in $\C$. Then it exists a stationary solution $(X_t)$ that is an exponentially
$\phi$-mixing processes, in the sense that the coefficients $\phi_r$ decay
exponentially fast, see \cite{Athreya86}.

\subsection{The MA($q$) model}
We consider now observations $(X_t)$ such that
$X_t=\sum_{j=1}^qb_j\epsilon_{t-j}$ for all $t\in\Z$. Assume that $\mathcal
B(z)=\sum_{j=1}^qb_jz^j$ has no root inside the unit disk in $\C$ so that
$(X_t)$ is invertible (admits an AR($\infty$) representation). By definition the
process $(X_t)$ is stationary and $\phi$-dependent - it is even $q$-dependent,
in the sense that $\phi_r = 0$ for $r>q$. Moreover it is bounded iff the innovations
are bounded. So this process satisfies the assumptions of Theorem \ref{mainthm}.

\subsection{Non linear models}
Consider an extension of the AR($p$) model of the form
\begin{equation}
X_t=F(X_{t-1},\ldots,X_{t-p};\epsilon_t),\qquad\forall t\in\Z.
\end{equation}
To prepare the general case we recall some material from \cite{Meyn1993}. Remember that
the observations are assumed to belong to the compact set $[-B,B]$. The
Lagrange stability, irreducibility and aperiodicity conditions hold when the innovations admits a lower
semi-continuous density on $[-B,B]$ and for any $|x|\le B$ we have 
$$
[-B,B]=A^+(x):=\{ F_k(x,w_1,\ldots,w_k); \quad
k\ge1,\,(w_1,\ldots,w_k)\in\mbox{Support}^k(\epsilon)\}
$$
with $F_k:\R^{k+1}\mapsto\R$ defined recursively by the relation
$F_{k+1}(\cdot,w)=F(F_k(\cdot),w)$, $F_1=F$. A direct application of Proposition
7.5 of \cite{Meyn1993} yields that $(X_t)$ is a T-chain (we refer to
\cite{Meyn1993} for the definition) if $F_1(x,w)$ is continuously differentiable
on $w$ and for each $x_0\in \R^p$ there exists $(w_k)_{1\le k\le p}$ such that 
$\partial F_k/\partial w_k(x_0,w_1,\ldots,w_k)\neq0$ for all $1\le k\le p$. For
example the generalized AR-GARCH models of the form $F(x,w)=R(x)+\sigma(x)w$
with $R$ and $\sigma>0$ continuously differentiable is a T-chain.\\

Assume that  $(X_t)$ is an irreducible, aperiodic, Lagrange stable T-chain. Then it satisfies the Doeblin
condition and is thus exponentially $\phi$-mixing, see Theorem 16.2.7 of
\cite{Meyn1993}.

\section{Implementation and simulations}
\label{section_algo}

\subsection{RJMCMC method}

The Gibbs estimator, given by~\eqref{gibbsdfn}, takes the form of an integral
over a large dimensional space. It can thus be computed by Monte Carlo methods.
This is actually a classical approach for Bayesian estimators, see e.g.
\cite{Christian2,ChristianOld}. Here, we use the RJMCMC algorithm - Reversible
Jumb Markov Chain Monte Carlo, \cite{RJMCMC}. This method is implemented for
example in \cite{AL} to compute a Gibbs estimator that takes exactly the same
form than ours.

\subsection{Simulations study in the AR case}

We compare here the Gibbs estimator given by~\eqref{gibbsdfn} to the ``classical
approach'' in the AR case. This approach, for example as implemented in the R
software (\cite{R}), computes the least square estimator in each
submodel AR$(p)$ and then selects the order $p$ by Akaike's AIC criterion \cite{aic}.

We generate the data according to the following models:
\begin{align}
\label{align1}
X_t & = 0.5 X_{t-1} + 0.1 X_{t-2} + \varepsilon_{t}
\\
\label{align2}
X_{t} & = 0.6 X_{t-4} + 0.1 X_{t-8} + \varepsilon_{t}
\\
\label{align3}
X_{t} & = \cos (X_{t-1})\sin(X_{t-2}) + \varepsilon_{t}
\end{align}
where $\varepsilon_{t}$ is the innovation. We will use two models for the innovation:
the uniform case, $\varepsilon_{t}\sim\mathcal{U}[-a,a]$, and the Gaussian case,
$\varepsilon_t \sim \mathcal{N}(0,\sigma^{2})$.
In the first case, the processes defined in~\eqref{align1}, \eqref{align2} and~\eqref{align3} satisfy the assumptions
of Theorem \ref{mainthm} (see Section~\ref{sectionexamples}) while the Gaussian case
is more classical in statistics, so it is worth testing if our method performs well in this context
too - even if our method does not receive any theoretical justification in this case, as it is
show in \cite{Doukhan1994} that autoregressive processes with gaussian noise are not $\phi$-mixing.
We take $\sigma=0.4$ and $a=0.70$ (In both cases this leads to ${\rm Var}(\varepsilon_{t})\simeq
0.16$).
The Gibbs estimator is used on all the possible AR models as in Example \ref{examar};
we fix $q=20$ and $\lambda = n/\hat{{\rm var}}(X)$, where $\hat{{\rm var}}(X)$ is the empirical
variance of the observed time series. We compare its performances to the ones of
AIC criterion as implemented in the R software and to the basic least square estimator
in the model $AR(q)$ - that we will call ``full model''. The experimental design is the following:
for each model, we simulate a time series of length $2n$, use the observations $1$ to $n$ as
a learning set and $n+1$ to $2n$ as a test set. We report the performances on the test set.
We take $n=100$ and $n=1000$ in the simulations. Each simulation is repeated 20 times, we report
on Table~\ref{tablesimu} the mean performance and standard deviation of each method.

\begin{table}[t!]
\caption{Performances of the Gibbs estimator, AIC and least square estimator in the full
model, on the simulations. Each simulation is repeated
20 times, we report on Table~\ref{tablesimu} the mean performance and standard deviation of
each method. We highlight the best result for each line. 
}
\label{tablesimu}
\begin{center}
\begin{tiny}
\begin{tabular}{|p{1.0cm}|p{1.0cm}|p{1.5cm}||p{1.6cm}|p{1.6cm}|p{1.6cm}|}
\hline
 $n$ & Model & Innovations & Gibbs & AIC & Full Model \\
\hline \hline
 $100$ & \eqref{align1} & unif.    & {\bf 0.165} (0.022) & {\bf 0.165} (0.023) & 0.182 (0.029) \\
       &                & Gaussian & 0.167 (0.023) & {\bf 0.161} (0.023) & 0.173 (0.027) \\
\hline
       & \eqref{align2} & unif.    & {\bf 0.163} (0.020) & 0.169 (0.022) & 0.178 (0.022) \\
       &                & Gaussian & {\bf 0.172} (0.033) & 0.179 (0.040) & 0.201 (0.049) \\
\hline
       & \eqref{align3} & unif.    & {\bf 0.174} (0.022) & 0.179 (0.028) & 0.201 (0.040) \\
       &                & Gaussian & {\bf 0.179} (0.025) & 0.182 (0.025) & 0.202 (0.031) \\
\hline \hline
 $1000$& \eqref{align1} & unif.    & {\bf 0.163} (0.005) & {\bf 0.163} (0.005) & 0.166 (0.005) \\
       &                & Gaussian & {\bf 0.160} (0.005) & {\bf 0.160} (0.005) & 0.162 (0.005) \\
\hline
       & \eqref{align2} & unif.    & {\bf 0.164} (0.004) & 0.166 (0.004) & 0.167 (0.004) \\
       &                & Gaussian & {\bf 0.160} (0.008) & 0.161 (0.008) & 0.163 (0.008) \\
\hline
       & \eqref{align3} & unif.    & {\bf 0.171} (0.005) & 0.172 (0.006) & 0.175 (0.006) \\
       &                & Gaussian & {\bf 0.173} (0.009) & {\bf 0.173} (0.009) & 0.176 (0.010) \\
\hline
\end{tabular}
\end{tiny}
\end{center}
\vspace*{-6pt}
\end{table}

It is interesting to note that our estimator performs better on Model~\eqref{align2}
and Model~\eqref{align3} while AIC performs slightly better on Model~\eqref{align1}.
The differences tends to
be less perceptible when $n$ grows - this is coherent with the fact that we develop
here a non-asymptotic theory. It is also interesting to note that our estimator seems
to work well even in the case of a Gaussian noise.

\section{Conclusion}

We proved that the Gibbs estimator can reach fast rates in the case of $\phi$-mixing
time series. It would now be interesting to extend this result to a more general
class of processes, e.g. weakly dependent ones. Note however the versions of
Bernstein's inequality known in the context of weak dependence (see e.g.
\cite{Dedecker2007a,devdep}) do not allow to reach this rate up to our knowledge.
More generally, the question of concentration of measure for time series is on a 
large part still open.

Another question is to provide a theoretical justification to our heuristic for the
tuining of $\lambda$ in practice.

\section{Proof of Theorem \ref{mainthm}}

\label{sectionproofs}

We start by a short overview of the proof. First, we state a result, Lemma \ref{exprisk},
that provides a control of the difference between the risk and the empirical risk of a predictor.
The main tool for the proof of this result is Samson's version of
Bernstein's inequality in Lemma \ref{lapmx}, that we remind in the appendix.
Lemma \ref{exprisk} is then used together with Donsker-Varadhan variational formula
(also reminded in the appendix, Lemma \ref{LEGENDRE}) to prove a PAC-Bayesian
type oracle inequality similar to the ones in \cite{Catoni2004}, Lemma \ref{PACBAYES},
that is the main tool used to prove Theorem \ref{mainthm}

\begin{lem}
\label{exprisk}
Under the hypothesis of Theorem \ref{mainthm}, we have, for any
$\theta\in\Theta(b+1)$, for any $0\le \lambda\le
(n-q)/[4(2+b)^2B^2\Phi^2 (q)]$,
\begin{equation*}
 \mathbb{E} \exp\left\{ \lambda  \left[
  \left(1-\frac{32 \Phi(q) \lambda(2+b)^2 B^2}{n-q}\right)
 \left(R(\theta)-R(\overline{\theta})\right) - r(\theta) + r(\overline{\theta})
      \right]\right\}
 \leq 1,
\end{equation*}
and
\begin{equation*}
 \mathbb{E} \exp\left\{ \lambda  \left[
  \left(1+\frac{32 \Phi(q) \lambda(2+b)^2 B^2}{n-q}\right)
 \left(R(\overline{\theta})-R(\theta)\right) - r(\overline{\theta}) + r(\theta)
      \right]\right\}
 \leq 1.
\end{equation*}
\end{lem}
\begin{proof}[Proof of Lemma \ref{exprisk}]
We apply Samson's version of Bernstein's inequality (see Lemma \eqref{lapmx} in the Appendix)
to $N=n-q$, $Z_i=(X_{i+1},\ldots,X_{i+q})$,
\begin{multline*}
f(Z_i)= \frac{1}{n-q}\Bigl[R(\theta)-R(\overline \theta)
\\
         - \left(X_{i+q}-f_{\theta}(X_{i+q-1},\dots,X_{i+1})\right)^{2} +
\left(X_{i+q}-f_{\overline{\theta}}(X_{i+q-1},\dots,X_{i+1})\right)^{2}
\Bigr].
\end{multline*}
Note that we have:
$$ S(f)=[R(\theta)-R(\overline \theta)- r(\theta)+r(\overline
\theta)],$$
and the $Z_i$ are uniformly mixing with coefficients $\phi_{r}^Z
= \phi_{\lfloor r/q \rfloor }$.
Note that $K_{\phi^Z}=1+\sum_{r=1}^{n-q}\sqrt{\phi_{\lfloor r/q \rfloor}}=
K^{(n)}_\phi(q)\leq \Phi(q)$.
For any $\theta$ and $\theta'$ in $\Theta$ let us put
$$ V(\theta,\theta') = \E_{\P}
\left\{\left[\Bigl(X_{q+1}-f_{\theta}(X_{q},...,X_{1})\Bigr)^{2}
-\Bigl(X_{q+1}-f_{\theta'}(X_{q},...,X_{1})\Bigr)^{2}\right]^{2}\right\}. $$
Noticing that $\sigma^2(f) \le V(\theta,\overline\theta) / (n-q)^2$ and that
$\|f\|_\infty\le 4(2+b)^2 B^2 / (n-q)$, for any $0\le \lambda\le
(n-q)/[4(2+b)^2B^2\Phi^2 (q)]$,
we have
$$
\ln \E_{\P}
\exp\left[\lambda
\Bigl(R(\theta)-R(\overline{\theta})-r(\theta)+r(\overline{\theta})\Bigr)\right]
\leq  \frac{8 \Phi(q) \lambda^{2}  V(\theta,\overline{\theta})}{n-q}.$$
Notice also that
\begin{multline*}
V(\theta,\overline{\theta})
= \E_{\P} \left\{
\left[2X_{q+1}-(f_{\theta}+f_{\overline{\theta}})(X_{q},...,X_{1})\right]^{2}
\left[(f_{\theta}-f_{\overline{\theta}})(X_{q},...,X_{1})\right]^{2}
\right\}
\\
\leq (2+\|\theta\|_{1}+\|\overline{\theta}\|_{1})^{2} B^{2} \E_{\P} \left\{
\left[(f_{\theta}-f_{\overline{\theta}})(X_{q},...,X_{1})\right]^{2}
\right\}
\\
= (2+\|\theta\|_{1}+\|\overline{\theta}\|_{1})^{2} B^{2}
\left[R(\theta)-R(\overline{\theta})\right]
\leq
4 (2+b)^{2} B^{2} \left[R(\theta)-R(\overline{\theta})\right]
\end{multline*}
as $\theta\in\Theta(b+1)$ and $\overline{\theta}\in\Theta(b)\subset\Theta(b+1)$.
This proves
the first inequality of Lemma~\ref{exprisk}. The second inequality is proved
exacly in the same way, but replacing $f$ by $-f$.
\end{proof}

We are now ready to state the following key result. Note that the very classical
definition of the Kullback divergence $\mathcal{K}(\rho,\pi)$ is reminded in
the appendix.
\begin{lem}[PAC-Bayesian oracle inequality for a $\phi$-mixing process]
\label{PACBAYES}
Under the hypothesis of Theorem \ref{mainthm}, we have, for any
$0\le \lambda\le
(n-q)/[4(2+b)^2B^2\Phi^2(q)]$, for any $0<\varepsilon<1$,
\begin{equation*}
\mathbb{P} \left\{
\begin{array}{l}
\forall \rho \in\mathcal{M}_{+}^{1}(\Theta(b+1)), \\
 \left(1-\frac{32 \Phi(q)\lambda(2+b)^2 B^2}{n-q}\right)
 \left(\int R {\rm d} \rho - R(\overline{\theta}) \right)\leq \int r {\rm d}\rho
      - r(\overline{\theta})
     + \frac{\mathcal{K}(\rho,\pi) +
\log\left(\frac{2}{\varepsilon}\right)}{\lambda}
\\
\text{ and }
\\
 \int r {\rm d} \rho - r(\overline{\theta})
    \leq \left(\int R {\rm d}\rho - R(\overline{\theta})\right)
          \left(1+\frac{32 \Phi(q)\lambda(2+b)^2 B^2}{n-q}\right)
     + \frac{\mathcal{K}(\rho,\pi) +
\log\left(\frac{2}{\varepsilon}\right)}{\lambda}
\end{array}
\right\}
\geq 1-\varepsilon.
\end{equation*}
\end{lem}
\begin{proof}[Proof of Lemma \ref{PACBAYES}]
 Let us fix $\varepsilon$, $\lambda$ and $\theta\in\Theta(b+1)$,
and apply the first inequality of
Lemma~\ref{exprisk}. We have:
\begin{equation*}
 \mathbb{E} \exp\left\{ \lambda  \left[
  \left(1-\frac{32 \Phi(q) \lambda(2+b)^2 B^2}{n-q}\right)
 \left(R(\theta)-R(\overline{\theta})\right) - r(\theta) + r(\overline{\theta})
      \right]\right\}
 \leq 1,
\end{equation*}
and we multiply this result by $\varepsilon/2$ and integrate it with
respect to $\pi_{b}({\rm d}\theta)$. Fubini's Theorem gives:
\begin{multline*}
 \mathbb{E} \int \exp\left\{ \lambda  \left[
  \left(1-\frac{32 \Phi(q)\lambda(2+b)^2 B^2}{n-q}\right)
 \left(R(\theta)-R(\overline{\theta})\right) - r(\theta) + r(\overline{\theta})+\log\Big(\frac \epsilon2\Big)
      \right]\right\} \pi_{b}({\rm d}\theta)
 \\
  \leq \frac{\varepsilon}{2}.
\end{multline*}
We apply Donsker-Varadhan variational formula (see Lemma~\ref{LEGENDRE} in the appendix)
and we get:
\begin{multline*}
 \mathbb{E}  \exp\Biggl\{ \sup_{\rho} \lambda  \Biggl[
  \left(1-\frac{32 \Phi(q)\lambda(2+b)^2 B^2}{n-q}\right)
 \left(\int R {\rm d}\rho-R(\overline{\theta})\right) - \int r {\rm d}\rho
 + r(\overline{\theta})+\log\Big(\frac \epsilon2\Big)
\\ - \mathcal{K}(\rho,\pi)
      \Biggr]\Biggr\}
  \leq \frac{\varepsilon}{2}.
\end{multline*}
As $e^{x}\geq \mathds{1}_{\mathbb{R}_{+}}(x)$, we have:
\begin{multline*}
\mathbb{P} \Biggl\{
\sup_{\rho} \lambda  \Biggl[
  \left(1-\frac{32 \Phi(q)\lambda(2+b)^2 B^2}{n-q}\right)
 \left(\int R {\rm d}\rho-R(\overline{\theta})\right) - \int r {\rm d}\rho
 + r(\overline{\theta})
      \\
+\log\Big(\frac \epsilon2\Big)\Biggr] - \mathcal{K}(\rho,\pi)\geq 0
 \Biggr\}
\leq \frac{\varepsilon}{2}.
\end{multline*}
Now, we follow the same proof again but starting with the second inequality of
Lemma~\ref{exprisk}. We obtain:
\begin{multline*}
\mathbb{P} \Biggl\{
\sup_{\rho} \lambda  \Biggl[
  \left(1+\frac{32 \Phi(q)\lambda(2+b)^2 B^2}{n-q}\right)
 \left(R(\overline{\theta})-\int R {\rm d}\rho\right) - r(\overline{\theta})
   + \int r {\rm d}\rho
\\
     +\log\Big(\frac \epsilon2\Big) - \mathcal{K}(\rho,\pi) \Biggr] \geq 0
 \Biggr\}
\leq \frac{\varepsilon}{2}.
\end{multline*}
A union bound ends the proof.
\end{proof}

We are now ready to give the proof of Theorem \ref{mainthm}.

\begin{proof}
First, we apply Lemma~\ref{PACBAYES}. From now, a work on the event
of probability at least $1-\varepsilon$ given by this lemma. In particular
we have
$\forall \rho \in\mathcal{M}_{+}^{1}(\Theta)$,
$$
\int R {\rm d} \rho - R(\overline{\theta})
\leq \frac{ \int r {\rm d}\rho
      - r(\overline{\theta})
     + \frac{\mathcal{K}(\rho,\pi) +
\log\left(\frac{2}{\varepsilon}\right)}{\lambda}
} {
 1-\frac{32 \Phi(q)\lambda(2+b)^2 B^2}{n-q}
}.
$$
For the sake of simplicity, during this proof, we will use the following
notation:
$$ C = 32 \Phi(q) (2+b)^2 B^2.$$
Taking $\rho = \hat{\rho}_{\lambda,b}$
leads to:
$$
\int R {\rm d} \hat{\rho}_{\lambda,b} - R(\overline{\theta})
\leq \frac{ \int r {\rm d}\hat{\rho}_{\lambda,b}
      - r(\overline{\theta})
     + \frac{\mathcal{K}(\hat{\rho}_{\lambda,b},\pi) +
\log\left(\frac{2}{\varepsilon}\right)}{\lambda}
} {
 1-\frac{\lambda C}{n-q}
}.
$$
We apply Lemma~\ref{LEGENDRE} to see that:
$$
\int R {\rm d} \hat{\rho}_{\lambda,b} - R(\overline{\theta})
\leq \inf_{\rho} \frac{ \int r {\rm d}\rho
      - r(\overline{\theta})
     + \frac{\mathcal{K}(\rho,\pi) +
\log\left(\frac{2}{\varepsilon}\right)}{\lambda}
} {
 1-\frac{\lambda C}{n-q}
}.
$$
Now, we use the second inequality of Lemma~\ref{PACBAYES} to see
that
\begin{multline}
\label{etape1}
\int R {\rm d} \hat{\rho}_{\lambda,b} - R(\overline{\theta})
\leq
\inf_{\rho} \frac{ \left(1+\frac{\lambda C}{n-q}\right)
\left( \int R {\rm d}\rho
      - R(\overline{\theta}) \right)
     + 2 \frac{\mathcal{K}(\rho,\pi) +
\log\left(\frac{2}{\varepsilon}\right)}{\lambda}
} {
 1-\frac{\lambda C}{n-q}
}
\\
\leq
\inf_{I\subset\{1,...,q\}} \inf_{\rho \ll \pi_{b}^{I}}
\frac{ \left(1+\frac{\lambda C}{n-q} \right)
\left( \int R {\rm d}\rho
      - R(\overline{\theta}) \right)
     + 2 \frac{  \mathcal{K}(\rho,\pi)  +
\log\left(\frac{2}{\varepsilon}\right)}{\lambda}
} {
 1-\frac{\lambda C}{n-q}
}
.
\end{multline}
By Jensen's inequality,
$$ \int R {\rm d} \hat{\rho}_{\lambda,b}  \geq R\left(
\hat{\theta}_{\lambda,b}\right) .$$
Also remark that, as soon as $\rho \ll \pi_{b}^{I}$,
\begin{multline*}
\mathcal{K}(\rho,\pi)
=
(|I|+1)\log(2) + \log {p \choose |I|} + \mathcal{K}(\rho,\pi_{b}^{I})
\\
\leq
(|I|+1)\log(2) + |I|\log \left(\frac{p{\rm e}}{|I|}\right) +
\mathcal{K}(\rho,\pi_{b}^{I})
\end{multline*}
(see, e.g., \cite{Catoni2003} page 190).
Now, for any $0<\delta<1$, for any
$I\subset\{1,...,p\}$, and
$\theta\in\Theta_{I}(B)$,
we take $\rho_{\delta,I,\theta}$ as the uniform measure on
$ \{ t\in\Theta_{I}(b): \|t-\theta\|_{1} \leq \delta \}$.
Note that as $\theta\in\Theta_{I}(B)$ and $\delta<1$, the support of
$\rho_{\delta,I,\theta}$
is included in $\Theta(b+1)$ the support of $\pi_b$. This is the reason why
$\pi_{b}$
is defined in this way. Inequality~\eqref{etape1}
leads to
\begin{multline*}
R\left( \hat{\theta}_{\lambda,b}\right)- R(\overline{\theta})
\\
\leq
\frac{1}{ 1-\frac{\lambda C}{n-q}}
\inf_{\delta>0} \inf_{I\subset\{1,...,q\}} \inf_{\theta\in\Theta_I (b)}
\Biggl\{ \left(1+\frac{\lambda C}{n-q} \right)
\left( R(\theta) + B^{2}\delta^{2}
      - R(\overline{\theta}) \right)
\\
     + 2 \frac{ (|I|+1)\log(2) + |I|\log \left(\frac{p{\rm e}}{|I|}\right)
+ |I|\log\left(\frac{b}{\delta}\right) +
\log\left(\frac{2}{\varepsilon}\right)}{\lambda}
\Biggr\}
\end{multline*}
and so, by choosing $\delta=\sqrt{|I|/(2B^2 \lambda)}$,
we get:
\begin{multline*}
R\left( \hat{\theta}_{\lambda,b}\right)- R(\overline{\theta})
\\
\leq
\frac{1}{ 1-\frac{\lambda C}{n-q}}
\inf_{\delta>0} \inf_{I\subset\{1,...,q\}} \inf_{\theta\in\Theta_I (b)}
\Biggl\{ \left(1+\frac{\lambda C}{n-q} \right)
\left( R(\theta)
      - R(\overline{\theta}) \right)
\\
   +  \frac{ |I|\left(B +
 2 \log\left(\frac{2B bp{\rm e}}{|I|}\sqrt{\frac{\lambda}{|I|}}\right)\right) +
2 \log\left(\frac{4}{\varepsilon}\right)}{\lambda}
\Biggr\}
\end{multline*}
Remember that $\lambda\le
(n-q) c$ where we put for short $c=1/[4(2+b)^2B^2 \Phi^{2}(q)]$. Let us take
$\lambda=\eta (n-q)/(2C) $ for some constant
$\eta$. Remark that $\eta  \leq 2cC$ ensures that $\lambda \le (n-q) c$ while
we need to impose $|I|< \eta B^{2}(n-q) / C$ in order to ensure that $\delta<1$.
We obtain:
\begin{multline*}
\P\Biggl\{
R(\hat{\theta}_{\lambda,b}) - R(\overline{\theta})
\leq
\inf_{
\tiny{
\begin{array}{c}
I\subset\{1,...,p\}
\\
|I|< \eta B^{2}(n-q)/C
\\
\theta \in \Theta_{I}(b)
\end{array}
}
}
\Biggl[
\left(\frac{2+\eta}{2-\eta}\right) \Bigl(R(\theta) - R(\overline{\theta}) \Bigr)
\\
+ \frac{2C}{(n-q)\eta}\left(|I| \left( B +
2 \log\left( \frac{Bbp{\rm e}}{|I|}\sqrt{\frac{2\eta
(n-q)}{|I|}}\right)\right)+2\log\left(\frac{2}{\varepsilon}\right)\right)
\Biggr]\Biggr\}
\geq 1-\varepsilon.
\end{multline*}
We end the computation by the remark that
$ \lambda = \eta(n-q)/(2C) = \eta(n-q) / [64 \Phi(q) (2+b)^2 B^2]$ and that
$\eta  \leq 2cC = 16/\Phi(q)$.
\end{proof}

\bibliography{biblio}

\begin{thebibliography}{37}
\providecommand{\natexlab}[1]{#1}
\providecommand{\url}[1]{\texttt{#1}}
\expandafter\ifx\csname urlstyle\endcsname\relax
  \providecommand{\doi}[1]{doi: #1}\else
  \providecommand{\doi}{doi: \begingroup \urlstyle{rm}\Url}\fi

\bibitem[Akaike(1973)]{aic}
H.~Akaike.
\newblock Information theory and an extension of the maximum likelihood
  principle.
\newblock In B.~N. Petrov and F.~Csaki, editors, \emph{2nd International
  Symposium on Information Theory}, pages 267--281. Budapest: Akademia Kiado,
  1973.

\bibitem[Alquier(2008)]{AlquierPAC}
P.~Alquier.
\newblock Pac-bayesian bounds for randomized empirical risk minimizers.
\newblock \emph{Mathematical Methods of Statistics}, 17\penalty0 (4):\penalty0
  279--304, 2008.

\bibitem[Alquier and Lounici(2011)]{AL}
P.~Alquier and K.~Lounici.
\newblock {PAC}-{B}ayesian bounds for sparse regression estimation with
  exponential weights.
\newblock \emph{Electronic Journal of Statistics}, 5:\penalty0 127--145, 2011.

\bibitem[Alquier and Wintenberger(2012)]{alqwin}
P.~Alquier and O.~Wintenberger.
\newblock Model selection for weakly dependent time series forecasting.
\newblock {\it {B}ernoulli} (to appear), available on arXiv:0902.2924, 2012.

\bibitem[Athreya and Pantula(1986)]{Athreya86}
K.~B. Athreya and S.~G. Pantula.
\newblock Mixing properties of {H}arris chains and autoregressive processes.
\newblock \emph{J. Appl. Probab.}, 23\penalty0 (4):\penalty0 880--892, 1986.
\newblock ISSN 0021-9002.

\bibitem[Audibert(2004)]{AudibertPhD}
J.-Y. Audibert.
\newblock Th\'eorie statistique de l'apprentissage: une approche
  pac-bay\'esienne.
\newblock HDR Universit\'e Paris VI, 2004.

\bibitem[Audibert(2010)]{AudibertHDR}
J.-Y. Audibert.
\newblock Pac-bayesian aggregation and multi-armed bandits.
\newblock HDR Universit\'e Paris Est, 2010.

\bibitem[Brockwell and Davis(2009)]{davis}
P.~Brockwell and R.~Davis.
\newblock \emph{Time Series: Theory and Methods (2nd Edition)}.
\newblock Springer, 2009.

\bibitem[B\"uhlmann and van~de Geer(2011)]{vdgb}
P.~B\"uhlmann and S.~van~de Geer.
\newblock \emph{Statistics for High-Dimensional Data}.
\newblock Springer, 2011.

\bibitem[Catoni(2003)]{Catoni2003}
O.~Catoni.
\newblock A pac-bayesian approach to adaptative classification.
\newblock \emph{Preprint Laboratoire de Probabilit\'es et Mod\`eles
  Al\'eatoires}, 2003.

\bibitem[Catoni(2004)]{Catoni2004}
O.~Catoni.
\newblock \emph{Statistical Learning Theory and Stochastic Optimization,
  Lecture Notes in Mathematics (Saint-Flour Summer School on Probability Theory
  2001, ed. J. Picard)}.
\newblock Springer, 2004.

\bibitem[Catoni(2007)]{Catoni2007}
O.~Catoni.
\newblock \emph{PAC-Bayesian Supervised Classification (The Thermodynamics of
  Statistical Learning)}, volume~56 of \emph{Lecture Notes-Monograph Series}.
\newblock IMS, 2007.

\bibitem[Cesa-Bianchi and Lugosi(2006)]{lugosi}
N.~Cesa-Bianchi and G.~Lugosi.
\newblock \emph{Prediction, Learning, and Games}.
\newblock Cambridge University Press, New York, 2006.

\bibitem[Dalalyan and Tsybakov(2008)]{DT1}
A.~Dalalyan and A.~Tsybakov.
\newblock Aggregation by exponential weighting, sharp {PAC}-{B}ayesian bounds
  and sparsity.
\newblock \emph{Machine Learning}, 72:\penalty0 39--61, 2008.

\bibitem[Dedecker et~al.(2007)Dedecker, Doukhan, Lang, Le\'on, Louhichi, and
  Prieur]{Dedecker2007a}
J.~Dedecker, P.~Doukhan, G.~Lang, J.~R. Le\'on, S.~Louhichi, and C.~Prieur.
\newblock \emph{Weak Dependence, Examples and Applications}, volume 190 of
  \emph{Lecture Notes in Statistics}.
\newblock Springer-Verlag, Berlin, 2007.

\bibitem[Donsker and Varadhan(1976)]{Donsker1975}
M.~D. Donsker and S.~S. Varadhan.
\newblock Asymptotic evaluation of certain markov process expectations for
  large time. iii.
\newblock \emph{Communications on Pure and Applied Mathematics}, 28:\penalty0
  389--461, 1976.

\bibitem[Doukhan(1994)]{Doukhan1994}
P.~Doukhan.
\newblock \emph{Mixing}, volume~85 of \emph{Lecture Notes in Statistics}.
\newblock Springer-Verlag, New York, 1994.

\bibitem[Gerchinovitz(2011)]{gerchi}
S.~Gerchinovitz.
\newblock Sparsity regret bounds for individual sequences in online linear
  regression.
\newblock In \emph{Proceedings of COLT'11}, 2011.

\bibitem[Green(1995)]{RJMCMC}
P.~J. Green.
\newblock Reversible jump markov chain monte carlo computation and bayesian
  model determination.
\newblock \emph{Biometrika}, 82\penalty0 (4):\penalty0 711--732, 1995.

\bibitem[Hamilton(1994)]{hamilton}
J.~Hamilton.
\newblock \emph{Time Series Analysis}.
\newblock Princeton University Press, 1994.

\bibitem[Ibragimov(1962)]{ib62}
I.~A. Ibragimov.
\newblock Some limit theorems for stationary processes.
\newblock \emph{Theory of Probability and its Application}, 7\penalty0
  (4):\penalty0 349--382, 1962.

\bibitem[Littlestone and Warmuth(1994)]{LiWa}
N.~Littlestone and M.K. Warmuth.
\newblock The weighted majority algorithm.
\newblock \emph{Information and Computation}, 108:\penalty0 212--261, 1994.

\bibitem[Marin and Robert(2007)]{Christian2}
J.-M. Marin and C.~P. Robert.
\newblock \emph{Bayesian Core: A practical approach to computational Bayesian
  analysis}.
\newblock Springer, 2007.

\bibitem[McAllester(1999)]{McA2}
D.~A. McAllester.
\newblock Pac-bayesian model averaging.
\newblock In \emph{Procs. of of the 12th Annual Conf. On Computational Learning
  Theory, Santa Cruz, California (Electronic)}, pages 164--170. ACM, New-York,
  1999.

\bibitem[Meir(2000)]{meir}
R.~Meir.
\newblock Nonparametric time series prediction through adaptive model
  selection.
\newblock \emph{Machine Learning}, 39:\penalty0 5--34, 2000.

\bibitem[Meyn and Tweedie(1993)]{Meyn1993}
S.~P. Meyn and R.~L. Tweedie.
\newblock \emph{Markov chains and stochastic stability}.
\newblock Communications and Control Engineering Series. Springer-Verlag London
  Ltd., London, 1993.
\newblock ISBN 3-540-19832-6.

\bibitem[Modha and Masry(1998)]{modha}
D.~S. Modha and E.~Masry.
\newblock Memory-universal prediction of stationary random processes.
\newblock \emph{IEEE transactions on information theory}, 44\penalty0
  (1):\penalty0 117--133, 1998.

\bibitem[{R Development Core Team}(2008)]{R}
{R Development Core Team}.
\newblock \emph{R: A Language and Environment for Statistical Computing}.
\newblock R Foundation for Statistical Computing, Vienna, 2008.

\bibitem[Robert(1996)]{ChristianOld}
C.~P. Robert.
\newblock \emph{M\'ethods de Monte Carlo par chaines de Markov}.
\newblock Economica (Paris), 1996.

\bibitem[Samson(2000)]{samson}
P.-M. Samson.
\newblock Concentration of measure inequalities for markov chains and
  $\phi$-mixing processes.
\newblock \emph{The Annals of Probability}, 28\penalty0 (1):\penalty0 416--461,
  2000.

\bibitem[Seldin et~al.(2011)Seldin, Laviolette, Cesa-Bianchi, Auer, and
  Shawe-Taylor]{Seldin}
Y.~Seldin, F.~Laviolette, N.~Cesa-Bianchi, P.~Auer, and J.~Shawe-Taylor.
\newblock \emph{PAC-Bayesian inequalities for martingales}.
\newblock arXiv:1110.6886, 2011.

\bibitem[Shawe-Taylor and Williamson(1997)]{STW97}
J.~Shawe-Taylor and R.~Williamson.
\newblock A pac analysis of a bayes estimator.
\newblock In \emph{Proceedings of the Tenth Annual Conference on Computational
  Learning Theory, COLT'97}, pages 2--9. ACM, 1997.

\bibitem[Stoltz(2010)]{stoltz}
G.~Stoltz.
\newblock Agr\'egation s\'equentielle de pr\'edicteurs : m\'ethodologie
  g\'en\'erale et applications \`a la pr\'evision de la qualit\'e de l'air et
  \`a celle de la consommation \'electrique.
\newblock \emph{Journal de la SFDS}, 151\penalty0 (2):\penalty0 66--106, 2010.

\bibitem[Tibshirani(1996)]{Tibshirani-LASSO}
R.~Tibshirani.
\newblock Regression shrinkage and selection via the lasso.
\newblock \emph{J. Roy. Statist. Soc. Ser. B}, 58\penalty0 (1):\penalty0
  267--288, 1996.

\bibitem[Tsybakov(2003)]{TsybakovAgg}
A.~Tsybakov.
\newblock Optimal rates of aggregation.
\newblock In B.~Sch\"olkopf and M.~K. Warmuth, editors, \emph{Learning Theory
  and Kernel Machines}, pages 303--313. Springer LNCS, 2003.

\bibitem[Vovk(1990)]{VOVK}
V.G. Vovk.
\newblock Aggregating strategies.
\newblock In \emph{Proceedings of the 3rd Annual Workshop on Computational
  Learning Theory (COLT)}, pages 372--283, 1990.

\bibitem[Wintenberger(2010)]{devdep}
O.~Wintenberger.
\newblock Deviation inequalities for sums of weakly dependent time series.
\newblock \emph{Electronic Communications in Probability}, 15:\penalty0
  489--503, 2010.

\end{thebibliography}

\appendix

\section{Samson's version of Bernstein's inequality and Donsker-Varadhan variational formula}

\begin{lem}[\cite{samson} (page 460, line7)]
\label{lapmx}
Let $N\in\mathbb{N}$.
Let $(Z_i)_{i\in\mathbb{Z}}$ be a stationary process, let $(\phi_{r}^Z)$ denote
its
$\phi$-mixing coefficients, let $f$ be a measurable
function $\mathbb{R}\rightarrow [-M,M]$  and let
$$ S_N (f) := \sum_{i=1}^{N} f(Z_i). $$
Then:
\begin{equation*}
\ln \E(\exp(\lambda( S(f)-\E S(f))))\le 8K_{\phi^Z} N \sigma^2(f)\lambda^2,
\mbox{ for all }0\le \lambda\le 1/(M K_{\phi^Z}^2) ,
\end{equation*}
where $K_{\phi^Z}=1+\sum_{r=1}^N \sqrt{\phi^{Z}_r}$ and
$ \sigma^{2}(f) = {\rm Var}\left[f(Z_i)\right] $.
\end{lem}

\begin{dfn}
Given a measurable space $(E,\mathcal{E})$ we let
$\mathcal{M}_{+}^{1}(E)$ denote the set of all probability measures
on $(E,\mathcal{E})$. The Kullback divergence is a pseudo-distance on
$\mathcal{M}_{+}^{1}(E)$ defined, for any
$(\pi,\pi')\in[\mathcal{M}_{+}^{1}(E)]^{2}$ by the equation
$$
\mathcal{K}(\pi,\pi')= \left\{
\begin{array}{l}
\pi[\log(d \pi/d \pi')] \quad \text{ if } \pi \ll \pi',
\\
\\
+ \infty \quad \text{ otherwise.}
\end{array}
\right.
$$
with the convention that $\pi[h] = \int h(x) \pi({\rm d}x) $ for any measurable
function $h$.
\end{dfn}
\begin{lem}[\cite{Donsker1975} variational formula]
\label{LEGENDRE}
For any $\pi$ in the set $\mathcal{M}_{+}^{1}(E)$,
for any measurable function $h:E\rightarrow\mathbb{R}$ such
that $ \pi[\exp (h)]<+\infty $ we have:
\begin{equation} \label{lemmacatoni}
\pi[\exp
(h)]=\exp\left(\sup_{\rho\in\mathcal{M}_{+}^{1}(E)}\biggl(\rho
[h]-\mathcal{K}(\rho,\pi)\biggr)\right),
\end{equation}
with convention $\infty-\infty=-\infty$. Moreover, as soon as $h$ is
upper-bounded on the support of $\pi$, the supremum with respect to
$\rho$ in the right-hand side is reached for the Gibbs measure
$\pi\{h\}$ defined by
$$ \pi\{h\}({\rm d}x) = \frac{ e^{h(x)} \pi({\rm d}x) }{ \pi[\exp(h)]}.  $$
\end{lem}

\end{document}